\documentclass[11pt]{article}
\usepackage{amsfonts}
\usepackage{amsmath,amscd,amsthm}
\usepackage{fullpage}
\usepackage{amssymb}
\usepackage{centernot} 
\usepackage{enumerate} 
\usepackage{pb-diagram} 
\usepackage{mathrsfs}
\usepackage[OT2,T1]{fontenc}
\usepackage{seqsplit}
\usepackage{tikz}
\usepackage{tikz-cd}
\usepackage{color}
\usepackage{array}
\usepackage{verbatim}
\usepackage{url}

\DeclareSymbolFont{cyrletters}{OT2}{wncyr}{m}{n}
\DeclareMathSymbol{\Sha}{\mathalpha}{cyrletters}{"58}

\definecolor{refkey}{rgb}{1,1,1}
\definecolor{labelkey}{rgb}{1,1,1}
\definecolor{cite}{rgb}{0.9451,0.2706,0.4941}
\definecolor{ruri}{rgb}{0.0078,0.4022,0.8010}

\usepackage[%
bookmarks=true,bookmarksnumbered=true,%
colorlinks=true,linkcolor=ruri,citecolor=red%
 ]{hyperref}

\makeindex 

\def\F{{\rm \mathbb{F}}}
\def\Z{{\rm \mathbb{Z}}}

\def\Q{{\rm \mathbb{Q}}}

\def\R{{\rm \mathbb{R}}}

\def\P{{\rm \mathbb{P}}}

\def\p{{\rm \mathfrak{p}}}
\def\m{{\rm \mathfrak{m}}}

\def\O{{\rm \mathcal{O}}}

\def\Nm{{\rm Nm}}

\def\Aut{{\rm Aut}}

\def\Cl{{\rm Cl}}

\def\Disc{{\rm Disc}}

\def\Gal{{\rm Gal}}

\DeclareMathOperator\rank{{rank}}

\def\End{{\rm End}}
\def\Sel{{\rm Sel}}

\def\tr{{\rm tr}}

\numberwithin{equation}{section}

\newtheorem{theorem}{Theorem}[section]
\newtheorem{lemma}[theorem]{Lemma}

\newtheorem{corollary}[theorem]{Corollary}
\newtheorem{proposition}[theorem]{Proposition}

\DeclareMathOperator{\Jac}{Jac}

\begin{document}
\title{Rank stability in quadratic extensions 
and \\Hilbert's tenth problem for the ring of integers of a number field 
}

\author{Levent Alp\"oge, Manjul Bhargava, Wei Ho, and Ari Shnidman}

\date{}

\makeatletter
\DeclareRobustCommand{\pmods}[1]{\mkern4mu({\operator@font mod}\mkern6mu#1)}
\makeatother

\maketitle

\begin{abstract}
We show that for any quadratic extension of number fields $K/F$, there exists an abelian variety $A/F$ of positive rank whose rank does not grow upon base change to $K$. 
This result implies that Hilbert's tenth problem over the ring of integers of any number field has a negative solution. That is, for the ring $\O_K$ of integers of any number field $K$, there does not exist an algorithm that answers the question of whether a polynomial equation in several variables over~$\O_K$ has solutions in~$\O_K$. 
\end{abstract}

\makeatletter
\makeatother

\section{Introduction}

In his celebrated 1900 address to the International Congress of Mathematicians, Hilbert~\cite{hilbertsproblems} discussed 23 major open problems of importance for 20th century mathematics.  In particular, his Problem \#10 called for an algorithm to determine whether an integer polynomial equation in several variables has solutions in~the integers. 

Hilbert's 10th problem was answered in the negative by the work of Davis, Putnam, and Robinson~\cite{davisputnamrobinson} and Matiyasevich~\cite{MR258744}. The question has since remained whether Hilbert's 10th problem still has a negative answer when the ring of rational integers is replaced by the ring of integers in some number field.

Poonen \cite{poonen-rankstability} and independently Cornelissen, Pheidas, and Zahidi \cite{CornelissenPheidasZahidi} proved that for a number field $K$, if there exists an elliptic curve defined over $\Q$ with $\rank E(\Q) = \rank E(K) = 1$, then Hilbert's 10th problem has a negative answer for $O_K$. 
This was strengthened independently by Poonen 
and by Shlapentokh \cite{Shlapentokh-rankstability} to show that the latter condition can be replaced by $\rank E(\Q) = \rank E(K) >0$. Shlapentokh \cite[Section 8]{MazurRubinInventiones2010} also showed that if, for every cyclic extension $K/F$ of prime degree, there exists an elliptic curve $E/F$ such that $\rank E(F) = \rank E(K) >0$, then Hilbert's tenth problem has a negative solution for the ring of integers of every number field. This was later further strengthened by Shlapentokh \cite[Theorems 3.1 and 4.8]{MazurRubinShlapentokh} to show that if, for every {\it quadratic} extension $K/F$ of number fields, there exists an {\it abelian variety} $A/F$ such that $\rank A(F) = \rank A(K) >0$, then Hilbert's tenth problem has a negative solution for the ring of integers of every number field. 

It is the hypothesis of the latter result of Shlapentokh that we prove in this paper: 

\begin{theorem}\label{thm: main}
 Let $K/F$ be any quadratic extension of number fields.
 Then there exists an abelian variety $A/F$ such that $\rank A(F) = \rank A(K) >0$. 
\end{theorem}

\pagebreak 

\begin{corollary}\label{cor: main}
Let $K$ be any number field.  Then $\Z$ has a diophantine model over~$\O_K$, and so 
Hilbert's tenth problem has a negative answer over $\O_K$.
That is, for the ring $\O_K$ of integers of any number field $K$, there does not exist an algorithm that answers the question of whether a polynomial equation in several variables over~$\O_K$ has solutions.
\end{corollary}

Mazur and Rubin \cite{MazurRubinInventiones2010} showed that Corollary \ref{cor: main} follows from the conjectural finiteness of the Tate-Shafarevich group; see also the work of Murty and Pasten~\cite{MurtyPasten} and Pasten \cite{Pasten2023} for two other conditional results. Unconditional proofs for particular classes of number fields were given by a number of authors over the years, including Denef \cite{Denef75, Denef80}, Denef and Lipshitz \cite{DenefLipshitz}, Pheidas \cite{Pheidas}, Shlapentokh \cite{Shlapentokh-Extension}, Videla \cite{Videla}, Shapiro--Shlapentokh \cite{ShapiroShlapentokh},  Garcia-Fritz--Pasten \cite{GFP2020}, Shnidman--Weiss \cite{ShnidmanWeissRankGrowth}, and Kundu--Lei--Sprung \cite{KunduLeiSprung}.  
In a beautiful recent preprint, Koymans and Pagano~\cite{KoymanPagano} give an independent unconditional proof of~Corollary~\ref{cor: main}, via a slightly different version of Theorem \ref{thm: main}, using  different methods. However, it is interesting to note that both their proof and ours make use of results in additive combinatorics in number fields, such as the recent work of Kai~\cite{WataruKai2023}, though in rather different ways.\footnote{In particular, the case of linear equations in {three} primes that we use may be considered ``classical'', in that it can also be deduced from the methods of Vinogradov~\cite{vinogradov}, van der Corput~\cite{vanderCorput}, and Mitsui~\cite{Mitsui,mitsui-siegel-walfisz}.}

Finally, we remark that the~methods~of Koymans and Pagano~\cite{KoymanPagano} involve $2$-Selmer groups of elliptic curves with full 2-torsion over $F$, while our proof uses $\ell$-isogeny Selmer groups of Jacobians $J_n$ of the hyperelliptic curves $C_n:y^2 = x^\ell + n$. 

\vspace{.095in}

We now describe the ideas behind our proof more precisely. Let $\ell$ be an odd prime that is unramified in $K=F(\sqrt{q})$, and let $\zeta$ denote a primitive $\ell$-th root of unity lying in an extension field of~$K$. 
If $A/F(\zeta)$ is an abelian variety with $\rank A(F(\zeta)) = \rank A(K(\zeta)) >0$, then we also have $\rank A'(F) = \rank A'(K) >0$, where $A'$ is the Weil restriction of $A$ from $F(\zeta)$ to $F$. 
Thus, for the purposes of proving Theorem~\ref{thm: main}, by replacing $K$ and $F$ by $K(\zeta)$ and $F(\zeta)$, respectively, we may assume without loss of generality that $\zeta\in F$. 

We consider the family of curves $C_n : y^2 = x^\ell + n$ where $n\in F$. Their Jacobians $J_n$ admit complex multiplication by $\zeta$ over $F$, and our tool for studying the ranks of these Jacobians $J_n$ is the $(1 - \zeta)$-Selmer group. A key realization is that the $(1-\zeta)$-Selmer groups of the abelian varieties $J_n$ admit a large set $\Sigma$ of {\it silent primes}, in the terminology of Mazur and Rubin~\cite{MazurRubinLarsen}; these are the primes of $F$ that are inert or ramified in $F(\sqrt n)$, and these primes remain silent in that they impose no local conditions in the definition of the said $(1-\zeta)$-Selmer group (see Lemma \ref{lem: inert primes are silent}). 

Using the methods of Mazur--Rubin~\cite{MazurRubinInventiones2010} and Yu~\cite{Yu2016}, we choose $n=q^\ell r^2$ so that the Jacobian variety $J_n$ has $(1-\zeta)$-Selmer rank $0$ and thus rank $0$. We may then replace $n$ by $nt^2$, with $t$ an arbitrary product of silent primes satisfying a congruence condition modulo a constant depending only on $n$, and the $(1-\zeta)$-Selmer group of $A:=J_n$ will not change (and thus will still vanish!).

This great flexibility in choosing $t$ and thus $n=q^\ell r^2 t^2$ allows us to {\it simultaneously} also produce a non-torsion rational point on the $q$-quadratic twist $A^q=J_{r^2t^2}$ of $A=J_{q^\ell r^2 t^2}$, by solving the $\Sigma$-unit equation $a + 2rb = 1$. The fact that there are many solutions $(a,b)$ to this $\Sigma$-unit equation follows, for example, from a simpler version of Vinogradov's circle method approach to the ternary Goldbach problem \cite{vinogradov}, as generalized to number fields by Mitsui~\cite{Mitsui}; it also follows from the recent work of Kai~\cite{WataruKai2023}. Once we have such a solution, we obtain an $F$-rational point on the Fermat curve $ax^{\ell} + 2ry^{\ell} = 1$. The latter curve covers $C_{r^2t^2}$, where $t=a^{(\ell-1)/2}b$. In this way, we obtain an $F$-rational point on the Jacobian $A^q=J_{r^2t^2}$ that is generically non-torsion. 

We thereby conclude that 
\[\rank A^q(K) = \rank A(F) + \rank A^q(F) = \rank A^q(F) > 0,\] as desired. 


\vspace{.095in}

This article is organized as follows. In Section~\ref{sec: Selmer}, we describe the arithmetic of the curves $C:y^2=x^\ell+1$ and their twists, including the definition of the $(1-\zeta)$-Selmer group and its silent primes.  We show how to produce an $\ell$-ic twist of $C$ having rank 0 using the methods of Mazur--Rubin~\cite{MazurRubinInventiones2010} and \cite{Yu2016}, while ensuring that its further $q$-quadratic twist has a non-torsion rational point, obtained by making use of a solution to the aforementioned $\Sigma$-unit equation where $\Sigma$ is the set of silent primes. In Section~\ref{sec:Sunit}, we show how to deduce the existence of infinitely many solutions to the desired $\Sigma$-unit equation. Finally, in Section~\ref{sec: final}, we combine the above ingredients together to prove Theorem~\ref{thm: main} and thus Corollary~\ref{cor: main}. 

{\begin{center}    
{\quote{{ {{
Hello darkness, my old friend\\
I've come to talk with you again\\
Because a vision softly creeping\\
Left its seeds while I was sleeping\\
And the vision that was planted in my brain\\
Still remains\\
Within the sound of silence.}}}
\vspace{.1in}\\
-- Simon and Garfunkel
}}
\end{center}}

\section{The curve $y^2 = x^\ell + 1$ and its twists}\label{sec: Selmer}

Let $\ell$ be an odd prime, and let $F$ be a number field containing the group $\mu_\ell$ of  $\ell$-th roots of unity. Let $\zeta \in \mu_\ell$ be a generator. For $n \in F^\times$, let $C_n$ be the smooth projective hyperelliptic curve of genus $(\ell-1)/2$ in weighted projective space $\P(1,(\ell-1)/2,1)$ having affine model $y^2 = x^\ell + n$. Note that for any $\lambda \in F^\times$, the curves $C_n$ and $C_{n\lambda^{2\ell}}$ are isomorphic.  The automorphism $(x,y) \mapsto (\zeta x, y)$ induces an action of $\mu_\ell$ on $C_n$, and hence an action of the ring $\Z[\zeta]$ on the Jacobian $J_n := \Jac(C_n)$.  

It will be important for us that the varieties $C_n$ and $J_n$ admit both quadratic twists and $\mu_\ell$-twists. For a quadratic extension $K = F(\sqrt{q})$ of $F$, the $K$-quadratic twist $C_n^K$ of $C_n$ is the curve $qy^2 = x^\ell + n$, which is also isomorphic to $C_{q^\ell n}$. Similarly, the $K$-quadratic twist $J^K_n$ of $J_n$ is isomorphic to $J_{q^\ell n}$. The $\mu_\ell$-twists of $C_n$ are the curves  $C_{nr^2}$, for $r \in F^\times$, as $C_{nr^2}$ becomes isomorphic to $C_n$ over $F(r^{1/\ell})$, and the $\mu_\ell$-twists of $J_{n}$ are the Jacobians $J_{nr^2}$.

\subsection{$(1-\zeta)$-Selmer groups}

Via the embedding $\Z[\zeta] \hookrightarrow \End(J_n)$, we view $1 - \zeta$ as a self-isogeny $\phi \colon J_n \to J_n$ of degree $\ell$. We study the corresponding Selmer groups $\Sel_\phi(J_n)$ below; see \cite{PoonenSchaeffer} for further background on these.

We embed $C \hookrightarrow J$ via $P \mapsto P - \infty$, using the base point $\infty = [1 \colon 0 \colon 0] \in C_n(F)$. The one-dimensional $\F_\ell$-vector space $J_n[\phi](\overline{F})  = \ker(\phi)(\overline{F})$ is generated by the point $(0,\sqrt{n})$, and the Galois action  $\Gal(\overline{F}/F) \to \Aut(J_n[\phi](\overline{F})) \simeq \F_\ell^\times$ on $J_n[\phi]$ is the quadratic character cutting out $ F(\sqrt{n})/F$.  

Set $T := H^1(F, J_n[\phi])$. For a prime $\p$, let $F_\p$ be the corresponding completion of $F$, and set  $T_\p := H^1(F_\p, J_n[\phi])$. Let $W_\p \subset T_\p$ be the image of the boundary map $J_n(F_\p) \to T_\p$ coming from the short exact sequence
\begin{equation}\label{eq: ses}
0 \to J_n[\phi] \to J_n \stackrel{\phi}{\longrightarrow} J_n \to 0.
\end{equation} 
The $\phi$-Selmer group $\Sel_\phi(J_n)$ is defined to be the kernel of the homomorphism
$$T \to \prod_\p T_\p/W_\p$$ induced by the product of restriction maps $T \to T_\p$. Note that we have ignored the archimedean places, since $F$ is totally complex.

\begin{lemma}\label{lem: Selmer 0 implies rank 0}
    If $\Sel_\phi(J_n) = 0$, then $\rank J_n(F) = 0$.
\end{lemma}

\begin{proof}
The long exact sequence associated to $(\ref{eq: ses})$ gives $J_n(F)/\phi J_n(F) \hookrightarrow \Sel_\phi(J_n)$, hence the usual Selmer group inequality $\rank_{\Z[\zeta]} J_n(F) \leq  \dim_{\F_\ell}\Sel_\phi(J_n)$. 
\end{proof}

\subsection{Silent primes}

Crucial to our method is the presence of so-called {\it silent primes}, i.e., primes $\p$ with $T_\p = 0$. 
\begin{lemma}\label{lem: inert primes are silent}
    Suppose $\p \nmid \ell$ is inert or ramified in $F(\sqrt{n})/F$. Then $T_\p = 0$. 
\end{lemma}
\begin{proof}
    Since $F$ contains $\mu_\ell$, we see that $J_n[\phi]$ is isomorphic to its own Cartier dual. Hence, local Tate duality gives  $H^2(F_\p, J_n[\phi]) \simeq H^0(F_\p,J_n[\phi])$, and the local Euler characteristic formula reads
    \[\#T_\p = \#H^0(F_\p, J_n[\phi])\cdot \#H^2(F_\p,J_n[\phi])= (\#J_n[\phi](F_\p))^2 = 1.\] The last equality follows because $\sqrt{n} \notin F_\p$.
\end{proof}

Note that silent primes exist if and only if $n$ is not a square in $F$. 

\subsection{Twists of rank zero}\label{subsec: rank 0} 
Suppose now that $K = F(\sqrt{q})$ is a quadratic extension of $F$ that is unramified at all primes $\p$ above $\ell$. We will consider curves of the form $ C_{q^\ell r^2}$ for $r \in \O_F$ and their Jacobians $J_{q^\ell r^2}$. 
Since $F(\sqrt{q^{\ell}r^2}) = F(\sqrt{q^\ell})$, we have $J_{q^\ell}[\phi] \simeq J_{q^\ell r^2}[\phi]$,  and we may view all Selmer groups $\Sel_\phi(J_{q^\ell r^2})$ as subspaces of the same ambient $\F_\ell$-vector space $T := H^1(F, J_{q^\ell}[\phi])$. 

Following the methods of Mazur and Rubin \cite{MazurRubinInventiones2010}, Yu~\cite{Yu2016} proved:
    \begin{lemma}[{\cite[Theorem 4]{Yu2016}}] \label{lem:yu}
        There exists $r \in \O_F$ such that $\Sel_\phi(J_{q^\ell r^2}) = 0$. Moreover, we may assume that $r\notin \p$ for all primes $\p$ that ramify in $K$. 
    \end{lemma}
    \begin{proof}
         The first sentence is Yu's theorem, his assumption $\Gal(f) \simeq S_n$ being equivalent to our assumption that $q$ is not a square in $F^\times$. The second sentence is also clear from his proof.  
    \end{proof}
In the sequel, fix $r \in \O_F$ as in Lemma \ref{lem:yu}. We will use the following notation for special sets of primes in $F$: 
\begin{itemize}
    \item $S'$ is the set of primes $\p$ dividing either $\ell!$ or the conductor of $J_{q^\ell r^2}$;
    \item $S \subset S'$ is the subset of $\p \in S'$ that are either above $\ell$ or split in $K/F$;
    \item $S_{\rm inert}$ is the infinite set of primes $\p$ that are inert in $K$;
    \item $\Sigma := S \cup S_{\rm inert}$.
\end{itemize}
\pagebreak

\noindent
Finally, let $\O_{F,\Sigma}^\times \subset F^\times$ be the set of $\Sigma$-units, i.e., those $x$ such that $x\in \O_{F,\p}^\times$ for all primes $\p \notin \Sigma$.

    \begin{lemma}\label{lem: Selmer rank 0}
Let $t \in \O_{F,\Sigma}^\times$ and suppose that $t \in F_\p^{\times \ell}$ for  all $\p \in S$.
        Then $\Sel_\phi(J_{q^\ell r^2t^2}) = 0$. 
    \end{lemma}
\begin{proof}
The Selmer groups $\Sel_{\phi}(J_{q^\ell r^2})$ and $\Sel_{\phi}(J_{q^\ell r^2t^2})$ live inside the common ambient space $T$; we will show that their corresponding local conditions $W_\p$ are equal inside $T_\p$ for all primes $\p$. If $\p \in S$, this is because the two curves are isomorphic over $F_\p$. Otherwise, if $\p$ is inert or ramified in $K$, then $T_\p = 0$ by Lemma \ref{lem: inert primes are silent}. In all other cases, we have $\p \nmid \ell$ and $\p$ is a prime of good reduction for $J_{q^\ell r^2}$ and  $J_{q^\ell r^2t^2}$, and thus $W_\p = H^1_{\mathrm{un}}(F_\p, J_{q^\ell}[\phi])$ for both. Thus $\Sel_{\phi}(J_{q^\ell r^2t^2}) = \Sel_{\phi}(J_{q^\ell r^2}) = 0$. 
\end{proof}

\subsection{Twists of positive rank}

We continue with the assumptions on $F$ and $K$ from \S\ref{subsec: rank 0}.  To produce rational points on certain $K$-quadratic twists $J_{r^2t^2}$, we will use solutions to the $\Sigma$-unit equation $x + 2ry = 1$. First we need to know that many such solutions exist.  For $a, b\in F^\times$, write $t_{a,b}\in F^\times / F^{\times \ell}$ for the class of $a^{\frac{\ell - 1}{2}} b$.

\begin{proposition}\label{prop: S-unit solutions}
 The set $$X:= \left\{t_{a,b}: a, b\in \O_{F, \Sigma}^\times, 1 = a + 2rb, a^{\frac{\ell - 1}{2}} b\in F_\p^{\times \ell}\text{ for all } \p\in S\right\}\subseteq F^\times / F^{\times \ell}$$ is infinite.
\end{proposition}

Granting Proposition \ref{prop: S-unit solutions}, whose proof we defer until Section \ref{sec:Sunit}, we will produce rational points on the curves $C_{r^2b^2a^{\ell-1}}$. To show that these rational points are not (usually) torsion points, we use the following lemma.

\begin{lemma}\label{lem: finitely many twists with torsion}
    Let $n \in F^\times$. There are finitely many $t \in F^\times/F^{\times\ell}$ such that $J_{nt^2}(F)_{\mathrm{tors}} \neq J_{nt^2}[\phi](F)$. 
\end{lemma}

\begin{proof}
A version of this lemma holds for the twists of any fixed abelian variety, but the proof in this case is simplified by observing that the Jacobians $J_{nt^2}$ have complex multiplication, hence have everywhere potentially good reduction \cite[\S5]{SerreTate}.  It follows that there is a uniform upper bound (depending only on $\ell$ and $F$) on the order $\#J_{nt^2}(F)_{\mathrm{tors}}$ of the torsion subgroup. Indeed, if $\p$ is a place of $F$ above a rational prime $p$, then we may pass to a totally ramified extension of $F_\p$ over which $J_{nt^2}$ attains good reduction \cite[p.\ 498]{SerreTate}. We then use the Weil bound over $\O_F/\p$ to bound the prime-to-$p$ part of $J_{nt^2}(F)_{\mathrm{tors}}$. Applying this to two different primes $p$ gives a bound on $\#J_{nt^2}(F)_{\mathrm{tors}}$. 
Thus, we may choose an integer $N$ which is a multiple of $\#J_{nt^2}(F)_{\mathrm{tors}}$, for all $t \in F^\times$. Let $\chi_t \colon \Gal(\overline{F}/F) \to \mu_\ell$ be the character $\sigma \mapsto (\sqrt[\ell]{t^2})^\sigma/\sqrt[\ell]{t^2}$,
so that there is an isomorphism of $\Z[\zeta][\Gal(\overline{F}/F)]$-modules $J_{nt^2}[N] \simeq J_n[N] \otimes \chi_t^{-1}$. We see that if $J_{nt^2}[N](F)$ contains a point which is not killed by $1 - \zeta$, then the character $\chi_t$ is a subrepresentation of $J_n[N]$. This shows that there are only finitely many $t \in F^\times/F^{\times \ell}$ such that $J_{nt^2}(F)_{\mathrm{tors}} \neq J_{nt^2}[\phi](F)$.  
\end{proof}

Over $\overline{\Q}$, the curves $C_n$ are $\mu_\ell$-quotients of the $\ell$-th Fermat curve $x^\ell + y^\ell = z^\ell$. We use a twisted version of this covering to produce rational points on $C_{r^2b^2a^{\ell-1}}$.

\begin{proposition}\label{prop: positive rank}
    For all but finitely many $t_{a,b} \in X$, the Jacobian $J_{r^2a^{\ell-1}b^2}$ has positive rank.
\end{proposition}
\begin{proof}
For all but finitely many $t_{a,b} \in X$, we have $J_{r^2a^{\ell-1}b^2}(F)_{\mathrm{tors}} = J_{r^2a^{\ell-1}b^2}[\phi]$, by Lemma \ref{lem: finitely many twists with torsion}. Suppose $(a, b)$ is such a pair. 
Let $\widetilde{C} \colon ax^\ell + 2rby^\ell = z^\ell$ be the twisted Fermat curve and consider the cover $\pi \colon \widetilde{C} \to C_{r^2a^{\ell-1}b^2}$ given on the affine patch where $z = 1$ by \[\pi(x,y) = (axy^{-1}, a^{\frac{\ell-1}{2}}(y^{-\ell} - rb)).\] 
Since $1 = a+2rb$, the point $P = (a, a^{\frac{\ell-1}{2}}(1-rb)) = \pi(1,1)$ lies in $C_{r^2a^{\ell-1}b^2}(F)$. Since $a \neq 0$, this is not a torsion point, and therefore $J_{r^2a^{\ell-1}b^2}$ has positive rank. 
\end{proof}

\section{Proof of Proposition \ref{prop: S-unit solutions}} \label{sec:Sunit}

In this section, we give a proof of Proposition \ref{prop: S-unit solutions}.  The argument is standard, but for convenience in citing the literature, we prove a much stronger result.

\begin{proposition}\label{prop: sum of two primes}
Let $F$ be a totally complex number field. Let $I\subseteq \O_F$ be an ideal. Let $C\in \Z^+$ be prime to $I$. Let $\beta \in 2\O_F$ be nonzero. Let $u_1, u_2, u_3\in \left(\O_F / (C)\right)^\times$ be such that $u_1 + \beta u_2\equiv u_3\pmods{C}$. Then there exist infinitely many triples $(p_1, p_2, p_3)\pmods{\O_F^\times}$ such that:
\begin{itemize}
\item $p_i\in I$ with $\Nm{(p_i)} / \Nm{(I)}$ prime;
\item $p_1 + \beta p_2 = p_3$; and
\item $p_i\equiv u_i\pmods{C}$.
\end{itemize}
\end{proposition}

\begin{proof}
This is a standard modification of Vinogradov's $1937$ circle method argument solving ternary Goldbach for sufficiently large integers \cite{vinogradov}, which was subsequently applied to three-term homogeneous linear equations in the primes by van der Corput \cite{vanderCorput}. The generalization of the method to number fields is due to Mitsui \cite{Mitsui}.

We briefly sketch the argument here. One first reduces to needing to estimate the following sum of trigonometric integrals, which counts the number of solutions to $p_1 + \beta p_2 = p_3$ of bounded height:
\begin{align*}
\frac{1}{\left(\Nm_{F/\Q}\,{C}\right)^3}\sum_{\alpha_1, \alpha_2, \alpha_3\in \O_F/(C)} &e\left(-\frac{\tr_{F/\Q}\left(\alpha_1 u_1 + \alpha_2 u_2 + \alpha_3 u_3\right)}{C}\right)\\& \int_{\xi\in \left(\mathfrak{d}_{F/\Q} I\right)^{-1}\otimes_\Z \R/\Z} d\xi\, S_N\left(\xi + \frac{\alpha_1}{C}\right) S_N\left(\beta \xi + \frac{\alpha_2}{C}\right) S_N\left(-\xi + \frac{\alpha_3}{C}\right),
\end{align*}\noindent
where
\begin{align*}
S_N(z) := \sum_{p\in I : ||p||_\infty\leq N} \Lambda_F\left((p) I^{-1}\right) e\left(\tr_{F/\Q}\left(z p\right)\right),
\end{align*}
and $\Lambda_F := \mu_F * \log$ is the von Mangoldt function of $F/\Q$, and $||x||_\infty := \max_{v\mid \infty} |x|_v$.

This may be broken up into major and minor arcs in the standard way (with the same height cutoff as in Vinogradov, up to implicit constants), as in \cite{Mitsui}. To treat the integral over the major arcs and compute the singular series, we apply Mitsui's number field generalization~\cite{mitsui-siegel-walfisz} of the Siegel-Walfisz theorem. For the integral over the minor arcs, we use the Vaughan identity and repeated applications of Cauchy-Schwarz (both for the integral of the remaining two-variable exponential sum over the minor arcs as well as for the 
Type II sum arising from Vaughan), as in \cite{Mitsui}. The additive character modulo~$C$ serves to translate the dual vector and thus only affects implicit constants, e.g., in the geometric series estimate on the minor arcs. The finite part of the singular series is (ignoring small primes)
$$\prod_{\p} \lim_{n\to \infty} \Nm(\p) (\Nm(\p) - 1)^{-3}  \#\{(p_1, p_2) : p_1, p_2, p_1 + \beta p_2\in (\O_F/\p^n)^\times\},$$
which is positive by Hensel lifting and the given hypotheses, yielding the desired result.

The proposition can also be deduced from a recent and far more general result of Kai~\cite{WataruKai2023}, which generalizes to number fields the Green--Tao--Ziegler theorem on simultaneous prime values of linear forms \cite{GreenTaoZiegler}. For this, let $W \subset I^2$ be the subgroup of $(x,y) \in I^2$ such that $x \equiv u_1$, $y \equiv u_2$, and $x + \beta y \equiv u_3$ modulo~$C$. Define linear maps  $\psi_i \colon W \to I$ by $\psi_1(x,y) = x$, $\psi_2(x,y) = y$, and $\psi_3(x,y) = x + \beta y$. By \cite[Theorem 13.1]{WataruKai2023}, there are infinitely many pairs   $(p_1,p_2) \in W$ such that the triple $(p_1,p_2,p_1 + \beta p_2)$ satisfies the three conditions of the proposition.
\end{proof}

\begin{proof}[Proof of Proposition $\ref{prop: S-unit solutions}$]

Let $\m$ be the conductor of $K/F$, so that $\m$ is divisible by a prime $\p$ if and only if $\p$ ramifies in $K/F$.  The extension $K/F$ corresponds via class field theory to a surjection $\psi \colon \Cl_F(\m) \to \Z/2\Z$ from the ray class group $\Cl_F(\m)$ of conductor $\m$.  A prime $\p \nmid \Disc(K/F)$  splits in $K$ if and only if  $[\p] \in \ker(\psi)$. 
 
 First suppose that $\m = 1$, so that $K/F$ is everywhere unramified. Let $I$ be a prime ideal (coprime to $S$ and $r\m$) such that $[I] \notin\ker(\psi)$. We say that $p \in I$ is an $I$-prime if $\Nm(p)/\Nm(I)$ is prime. Note that any $I$-prime $p$ is an $S_{\rm inert}$-unit, and in particular a $\Sigma$-unit. Let $\gamma$ be any element of $\O_F$ such that $\gamma \in \p$ if and only if $\p \in S$.
 We will apply Proposition \ref{prop: sum of two primes} to the element $\beta = 2r\gamma^{n\ell}$ (for some large integer $n$ to be chosen below as needed) and the ideal $I$. This gives us $I$-primes $p_1,p_2,p_3$ such that $p_1 + 2r\gamma^{n\ell} p_2 = p_3$. We may and will assume that $p_2 \equiv p_3 \pmod{\p^n}$, for all $\p \in S$. Now take $a = p_1/p_3$ and $b = \gamma^{n\ell} p_2/p_3$. These are $\Sigma$-units and they satisfy $a + 2rb = 1$ by construction. It remains to check that 
 \[a^{\frac{\ell-1}{2}}b = (1-2rb)^{\frac{\ell-1}{2}} \gamma^{n \ell} p_2/p_3\]
 lies in $F_\p^{\times\ell}$ for all $\p \in S$. Since we have chosen $n$ large, the element $1 - 2rb$ is $\p$-adically close to $1$ and hence is an $\ell$-th power. Similarly, since $n$ is large and  $p_2/p_3 \equiv 1 \pmod{\p^n}$, the unit $p_2/p_3$ is an $\ell$-th power in $F_\p^\times$. Thus $a^{\frac{\ell-1}{2}}b$ is indeed an $\ell$-th power in $F_\p^\times$, completing the proof in this case.

 We now consider the case where $\m \neq 1$, i.e., where $K/F$ is not a subfield of the Hilbert class field~$H/F$. Applying Chebotarev's Theorem to the compositum $HK/F$, we see that there exist prime elements $p_i$ of $\O_F$ that are inert in $K$. Indeed, the restriction of $\psi$ to $\Gal(H_\m/H) \simeq (\O_F/\m)^\times/\pi_{\m}(\O_F^\times)$ has kernel an index $2$ subgroup $G_0$, and the principal primes~$p_i$ that are inert in $K/F$ are those whose images modulo $\m$ do not lie in $G_0$. This amounts to a squareclass condition on the primes~$p_i$ modulo each prime $\p \mid \m$. We now apply Proposition \ref{prop: sum of two primes} to the element $\beta = 2r\gamma^{n\ell}$ and the ideal $I = \O_F$. We define $a = p_1 / p_3$ and $b = \gamma^{n\ell} p_2 / p_3$ and assume that $p_2 \equiv p_3 \pmod{\p^n}$ for all $\p \in S$. Additionally, we prescribe the squareclasses of the $p_i$ modulo each prime dividing $\m$ so that the $p_i$ are inert in $K/F$. Since $K/F$ is unramified at primes above $\ell$, the sets $S$ and $\{\p \colon \p \mid \m\}$ are disjoint, so the new congruence conditions will not contradict the previous ones. We need only check that the equation $$t_1x^2 + 2r\gamma^{n\ell} t_2y^2 = t_3 z^2$$ has a solution modulo $\m$ for some choice of units $t_1,t_2,t_3$ modulo $\m$. Since we are free to assume $t_1 = t_3$, there is no local obstruction at $2$-adic primes $\p \mid \m$. For other primes $\p \mid \m$, the equation defines a smooth conic modulo $\p$ and hence is unobstructed as well.      

 As for infinitude, for each triple $(p_1, p_2, p_3)$ produced by the above construction, the ideal $$\mathfrak{t}_{p_1, p_2, p_3} := \left(p_1^{\frac{\ell - 1}{2}} p_2 \:\; p_3^{\frac{\ell - 1}{2}} \right)$$ has large prime factors occurring with multiplicity strictly smaller than $\ell$, and moreover the collection of such ideals produced by said construction has infinite support, else we would be producing infinitely many solutions to an ${\mathcal{S}}$-unit equation for a finite set ${\mathcal{S}}$ of primes, contradicting the theorem of Siegel and Mahler. It follows that the set of $\ell$-th power classes of the ideals $$\left(a^{\frac{\ell - 1}{2}} b \right) = \mathfrak{t}_{p_1, p_2, p_3} (\gamma^n p_3^{-1})^\ell$$ is infinite, whence the set of $\ell$-th power classes of the elements $a^{\frac{\ell - 1}{2}} b$ is infinite as well.
\end{proof}

\section{Proof of the main theorem (Theorem~$\ref{thm: main}$)}\label{sec: final}

In this section, for any quadratic extension $K/F$ of number fields, we construct an abelian variety $A/F$ with the property that $\rank A(F) =\rank A(K)>0$.

Let $\ell$ be an odd prime not dividing the discriminant of $K$. Then $K$ and the $\ell$-th cyclotomic field $\Q(\zeta_\ell)$ are linearly disjoint inside a common algebraic closure. By the Weil restriction construction, it is enough to prove the theorem for the quadratic extension $K(\zeta_\ell)/F(\zeta_\ell)$. Hence we may assume that $\zeta_\ell \in F$ and $K/F$ is unramified at primes above $\ell$, i.e., we are in the setting of \S\ref{subsec: rank 0}.  

Let $K = F(\sqrt{q})$, and choose $r$ as in Lemma \ref{lem:yu}. Then choose $a,b$ as in Proposition \ref{prop: positive rank}, so that  $\rank J_{r^2a^{\ell-1}b^2}(F) > 0$. Lemmas~\ref{lem: Selmer 0 implies rank 0} and \ref{lem: Selmer rank 0}  imply that $J_{q^\ell r^2a^{\ell-1}b^2}(F)$ has rank~$0$. It follows that 
\[\rank J_{r^2a^{\ell-1}b^2}(K) = \rank J_{r^2a^{\ell-1}b^2}(F) + \rank J_{q^\ell r^2a^{\ell-1}b^2}(F) = \rank J_{r^2a^{\ell-1}b^2}(F) > 0,\]
so that $A := J_{r^2a^{\ell-1}b^2}$ over $F$ is the sought-after abelian variety.

\vspace{.095in}
We have proven Theorem \ref{thm: main}, and therefore Corollary~\ref{cor: main} also  follows.

 \subsection*{Acknowledgments} We thank Tim Dokchitser, Vladimir Dokchitser, Peter Koymans, Jef Laga, Robert Lemke Oliver, Barry Mazur, Adam Morgan, Carlo Pagano, Hector Pasten, Bjorn Poonen, Arul Shankar, and Alexandra Shlapentokh for helpful conversations and comments. LA was supported by NSF DMS-2002109 and the Society of Fellows. MB was partially supported by a Simons Investigator Grant and NSF DMS-1001828. WH was partially supported by NSF DMS-2309115, the Minerva Research Foundation, and a grant from the Institute for Advanced Study.  AS was partially funded by the European Research Council (ERC, CurveArithmetic, 101078157), as well as the Ambrose Monell Foundation while at the Institute of Advanced Study.

\bibliographystyle{amsalpha}
\bibliography{references}

\end{document}